\documentclass[12pt,reqno]{amsart}

\usepackage{amsmath,amsfonts,amsthm,amssymb,amsxtra}
\usepackage{bbm} 
\usepackage{hyperref}	
\usepackage{enumerate}

\usepackage{xfrac}
\usepackage{todonotes, relsize, pgfplots, tikz, float}
\usetikzlibrary{calc,decorations.markings}
\tikzset{
	> = stealth,
	every pin/.style = {pin edge = {}},
	flow/.style = {decoration = {markings, mark=at position #1 with {\arrow{>}}},
		postaction = {decorate}
	},
	flow/.default = 0.5,
	main/.style = {color=#1, line width=0.5pt, line cap=round, line join=round},
	main/.default = black,
	fontscale/.style={font=\relsize{#1}},
}

\newtheorem{theorem}{Theorem}[section]
\newtheorem{proposition}[theorem]{Proposition}
\newtheorem{lemma}[theorem]{Lemma}
\newtheorem{corollary}[theorem]{Corollary}

\theoremstyle{definition}

\theoremstyle{remark}

\newtheorem{remark}[theorem]{Remark}
\newtheorem{remarks}[theorem]{Remarks}

\numberwithin{equation}{section}

\newcommand{\diff}{\,\mathrm{d}}
\renewcommand{\epsilon}{\varepsilon}

\newcommand{\N}{\mathbb{N}}

\renewcommand{\phi}{\varphi}
\newcommand{\R}{\mathbb{R}}

\newcommand{\Sph}{\mathbb{S}}

\DeclareMathOperator{\dist}{dist}

\DeclareMathOperator{\Ric}{Ric}

\DeclareMathOperator{\vol}{vol}

\usepackage{mathtools, todonotes, thmtools}\usepackage{tikz}
\usetikzlibrary{datavisualization}
\usetikzlibrary{datavisualization.formats.functions}
\usetikzlibrary{patterns}

\let\oldtocsection=\tocsection

\let\oldtocsubsection=\tocsubsection

\let\oldtocsubsubsection=\tocsubsubsection

\renewcommand{\tocsection}[2]{\hspace{0em}\oldtocsection{#1}{#2}}
\renewcommand{\tocsubsection}[2]{\hspace{1em}\oldtocsubsection{#1}{#2}}
\renewcommand{\tocsubsubsection}[2]{\hspace{2em}\oldtocsubsubsection{#1}{#2}}

\usepackage{todonotes}
\setuptodonotes{color=white, bordercolor=white, textcolor=blue, size=\normalsize}
\newif\ifTodo
\Todotrue

\begin{document}

	\begin{titlepage}
		\huge \title[An almost-almost-Schur lemma]{An almost-almost-Schur lemma on the 3-sphere}
		\vspace{7cm}
	\end{titlepage}
		
	\author{Tobias K\"onig}
	\address[Tobias K\"onig]{Institut für Mathematik, 
Goethe-Universität Frankfurt, 
Robert-Mayer-Str. 10, 60325 Frankfurt am Main, Germany}
\email{koenig@mathematik.uni-frankfurt.de}
	
	\author{Jonas W.~Peteranderl}
	
	\address[Jonas W.~ Peteranderl]{Mathematisches Institut, Ludwig-Maximilians-Universit\"at M\"unchen, Theresienstr.~39, 80333 M\"unchen, Germany}	\email{peterand@math.lmu.de}
	
	\begin{abstract}

       In the conformal class of the standard metric on the $3$-sphere, we prove a quantitative refinement of the Andrews--De~Lellis--Topping inequality in terms of a two-term distance to the set of minimizing conformal factors. This inequality is itself a stability result for the well-known Schur lemma and is therefore referred to as \textit{almost-Schur lemma}. Hence, our stability result may be viewed as an \textit{almost-almost-Schur lemma}.

        As a consequence, we deduce via interpolation the quantitative stability of an entire family of nonlinear Yamabe-type inequalities, including an inequality for the total volume-normalized $\sigma_2$-curvature $\mathcal F_2$. This extends a recent result by Frank and the second author for $d > 4$ to the case $d=3$. While the standard metric minimizes $\mathcal F_2$ if $d > 4$, it maximizes $\mathcal F_2$ if $d=3$. This is the main challenge in treating the case $d=3$ as it turns the related functional inequality into a \textit{reverse} Sobolev-type inequality.
\end{abstract}

\date{January 5, 2026}
\thanks{\copyright\, 2026 by the authors. This paper may be reproduced, in its entirety, for non-commercial purposes. 
}
	
	\maketitle
	\setcounter{page}{1}

\section{Introduction and main results}

\subsection{An almost-Schur lemma}
\label{subsec_almost_schur}

The well-known Schur lemma states that if a Riemannian manifold $(M,g)$ of dimension $d \geq 3$ is Einstein, that is,  
if its Ricci curvature tensor $\Ric$ and its scalar curvature $R$ satisfy $$\Ric = \frac{R}{d} g 
\,,$$ then $R$ must be constant.

De Lellis and Topping \cite{DeLellis2012} proved that for every closed Riemannian manifold $(M,g)$ of dimension $d \geq 3$ with non-negative Ricci curvature, one has 
\begin{equation}
    \label{eq:ADT_intro}
    \int_M \left|\Ric - \frac{\overline{R}}{d} g\right|^2  \diff \vol_g \leq \frac{d^2}{(d-2)^2} \int_M  \left|\Ric - \frac{R}{d} g\right|^2 \diff \vol_g\,,
\end{equation}
where $\overline{R} \coloneqq \vol(g)^{-1} \int_M R \diff\vol_g$ and $\diff \vol_g$ denotes the volume form on $(M,g)$. The constant is best possible. This inequality is a quantitative refinement of Schur's lemma and was therefore termed \emph{almost-Schur lemma} by the authors of \cite{DeLellis2012}; see also \cite[Corollary B.20]{ChowLuNi2006} for an independent but less known version of this result by Andrews.

Ge and Wang observed in \cite{Ge2012,Ge2013b} that the same refinement of Schur's lemma remains valid in dimension $d=3,4$ if instead of the Ricci curvature only the scalar curvature is assumed to be nonnegative. Their proof relies on a reformulation of the inequality in terms of the \textit{$\sigma_k$-curvatures} for $k=1,2$.

\subsection{An extension by Ge and Wang}

Let $(M,g)$ be a $d$-dimensional compact Riemannian manifold with $g$ conformally equivalent to a given background metric $g_*$. As an extension of the scalar curvature, Viaclovsky \cite{Viaclovsky2000} introduced a family of scalar curvatures, known as $\sigma_k$-curvatures, which are given by the $k$-th elementary symmetric functions of the Schouten tensor. We are interested in the cases $k=1,2$. In terms of the scalar and Ricci curvature, they have the form 
\begin{equation*}
    \sigma_1^g = \frac{1}{2(d-1)}R^g\qquad \text{and} \qquad \sigma_2^g = \frac{1}{2(d-2)^2} \left( \frac{d}{4(d-1)} (R^g)^2  - |\Ric^g|^2\right). 
\end{equation*}

The total $\sigma_1$-curvature and total $\sigma_2$-curvature, normalized by volume, are defined as
\begin{equation*}
    \mathcal F_1[g] \coloneqq \frac{1}{\operatorname{vol}(g)^\frac{d-2}{d}} \int_{M} \sigma_1^g \, \mathrm d \operatorname{vol}_g \qquad \text{and}\qquad  \mathcal F_2[g] \coloneqq \frac{1}{\operatorname{vol}(g)^\frac{d-4}{d}} \int_{M} \sigma_2^g \, \mathrm d \operatorname{vol}_g \,,
\end{equation*}
 respectively. These geometric quantities are conformally invariant, in the sense that $$\mathcal F_1[\Psi^*g]=\mathcal F_1[g]\qquad \text{and}\qquad \mathcal F_2[\Psi^*g]=\mathcal F_2[g]$$ for all conformal diffeomorphisms $\Psi$ of $(M, g_*)$, also known as \textit{M\"obius transformations}. It is easily checked that \eqref{eq:ADT_intro} can be equivalently stated as 
\begin{equation}
    \label{eq:ADT_intro_sigma}
   (\mathcal F_1[g])^2 \geq \frac{2d}{d-1} \mathcal F_2[g]\,.
\end{equation}

We now specialize to the case $d=3$. In this case, Ge and Wang \cite{Ge2013b} showed that the inequality remains valid even under the weaker requirement that the scalar curvature is nonnegative, and equality holds if and only if $g$ equals $g_*$ up to M\"obius transformations. Under stronger assumptions, inequality \eqref{eq:ADT_intro_sigma} was already known as part of a larger family of conformal quermassintegral-type inequalities in \cite{Guan2004}. 

Our goal in the present paper is to prove a quantitative version of the almost-Schur lemma for $M=\mathbb S^3$ and $g_*$ the standard round metric. Put differently, we ask whether a positive scalar curvature metric $g$ that is conformal to $g_*$ is almost $g_*$ (up to M\"obius transformations) if we have almost equality in \eqref{eq:ADT_intro_sigma}. Since we are restricting to conformal metrics, it is only natural to consider notions of closeness that involve the conformal factor. As the conformal factor is a (positive) function on the sphere, our next step is to introduce yet another formulation of \eqref{eq:ADT_intro_sigma} as integrals over the conformal factor.

\subsection{Functional formulation on $\mathbb S^3$} The metric $g$ is conformally equivalent to $g_*$ if it can be written as the product of a positive, smooth function and $g_*$. A direct computation shows that the parametrization $g=v^{4k/(d-2k)}g_*$, $v>0$, turns the total $\sigma_k$-curvature into a $2k$-homogeneous functional. For $k=1$, this gives the well-known $H^1$-norm in the Sobolev inequality. 

Let $M=\mathbb S^3$ as before. Since the compactness argument in our proof of stability relies on a reduction to the Yamabe inequality $\mathcal F_1[g] \geq \mathcal F_1[g_*]$ (for metrics $g$ conformal to $g_*$) similar to \cite{Frank2024b}, we consider two different parametrizations
$$g=w^{4}g_*\qquad \text{and}\qquad g=u^{-8}g_*$$
for the $\sigma_1$-curvature inequality and the $\sigma_2$-curvature inequality, respectively. As discussed, the functions $w$ and $u$ are smooth and positive on $\mathbb S^3$.
It is well known that the two components of $\mathcal F_1$ transform as 
\[ \operatorname{vol}(w^4 g_*) = \int_{\Sph^3} w^6 \, \mathrm d \omega \quad \text{ and } \quad \int_{\Sph^3} \sigma_1^{w^4 g_*} \, \mathrm d \operatorname{vol}_{w^4 g_*} = 2 \int_{\Sph^3} \left( |\nabla w|^2 + \frac{3}{4} w^2\right) \, \mathrm d \omega\,, \] where $\mathrm d\omega$ is the volume form and $\nabla$ the covariant derivative on $(\mathbb S^3,g_*)$.
Hence, we obtain
\begin{equation*}
    \label{F1[w] definition}
    \mathcal F_1[w^4 g_*] = \frac{2}{\left( \int_{\Sph^3} w^6 \, \mathrm d \omega \right)^\frac{1}{3}} \int_{\Sph^3} \left( |\nabla w|^2 + \frac{3}{4} w^2 \right) \, \mathrm d\omega \eqqcolon F_1[w]\,. 
\end{equation*}
While $\mathcal F_1[g]$ is invariant under M\"obius transformations, which act on the metric $g$ via pullback $\Psi^*g$, this translates for $F_1[w]$ to invariance under M\"obius transformations, which act on $w$
via $$[w]_\Psi\coloneqq J_\Psi^{\frac 16}w\circ \Psi\,,$$ where $J_\Psi$ denotes the Jacobian of $\Psi$. Thus, conformal invariance means for the functional $F_1$ that $F_1[[w]_\Psi]=F_1[w]$.
As a consequence, we can write the Yamabe inequality as \begin{equation}\label{eq:Sob}
    F_1[w]\geq F_1[1]\,,
\end{equation} and equality holds if and only if $w=\lambda [1]_\Psi$ for some $\lambda>0$ and some M\"obius transformation $\Psi$. Note that \eqref{eq:Sob} is nothing else but the (sharp) Sobolev inequality after stereographic projection.

Turning to the functional reformulation of $\mathcal F_2$, a more lengthy but similar computation shows that
\[ \operatorname{vol}(u^{-8} g_*) = \int_{\Sph^3} u^{-12} \, \mathrm d \omega \quad \text{ and } \quad \int_{\Sph^3} \sigma_2^{u^{-8} g_*} \, \mathrm d \operatorname{vol}_{u^{-8} g_*} = \int_{\Sph^3} e_2(u) \, \mathrm d \omega\,,  \]
where 
\begin{equation}
    \label{eq:e2(u)_definition}
    e_2(u) \coloneqq -64 \left(\sigma_1(u) + \frac{1}{2} |\nabla u|^2 + \frac{1}{32} u^2 \right) |\nabla u|^2 + \frac{3}{4} u^4\,
\end{equation} and 
\begin{equation}
    \label{eq:sigma_1(u)}
    \sigma_1(u) \coloneqq \frac{1}{8}\Delta (u^2) - |\nabla u|^2 + \frac{3}{32} u^2 = \frac 1{16} u^{-6}\sigma_1^{u^{-8}g_*}\,; 
\end{equation}
see \cite{Case2020}. The latter equality follows by direct computation and tells us that $\sigma_1(u)$ can be regarded as scalar curvature (up to a positive function). Let us emphasize that the sign -- in comparison to \cite{Frank2024b} -- changed in front of the bracket in \eqref{eq:e2(u)_definition} and the second derivative in \eqref{eq:sigma_1(u)} changed. This observation has severe implications for the behavior of optimizing sequences.
{We then obtain}
\begin{equation*}
    \label{F2[u] definition}
    \mathcal F_2[u^\frac{8}{d-4} g_*] =  \left( \int_{\Sph^3} u^{-12} \, \mathrm d \omega \right)^{\frac{1}{3}} \int_{\Sph^3} e_2(u) \, \mathrm d \omega \eqqcolon F_2[u] \,.
\end{equation*}
Similarly to the Yamabe inequality but with other exponents due to the difference in parametrization, a M\"obius transformation acts on $u$ via $$(u)_\Psi\coloneqq J_{\Psi}^{-\frac 1{12}} u\circ \Psi\,,$$ and conformal invariance becomes
$F_2[(u)_\Psi]=F_2[u]$.

In summary, using $w=u^{-2}$, inequality \eqref{eq:ADT_intro_sigma} 
can be expressed as 
\begin{equation}
    \label{eq:ADT_intro_functional}
    \tag{ADT}
    \frac{F_2[u]}{F_1[u^{-2}]^2} \leq \frac{F_2[1]}{F_1[1]^2} \qquad \text{ for all $u \in C^\infty(\mathbb S^3)$ with $u > 0$ and $\sigma_1(u) > 0$\,, }
\end{equation}
and equality holds in \eqref{eq:ADT_intro_functional} if and only if $u = \lambda (1)_\Psi$ for $\lambda>0$ and $\Psi$ Möbius transformation. Because of its equivalence with \eqref{eq:ADT_intro} (via \eqref{eq:ADT_intro_sigma}) and in view of the discussion in Subsection \ref{subsec_almost_schur}, we shall refer to \eqref{eq:ADT_intro_functional} as \emph{Andrews--De~Lellis--Topping inequality}.

\subsection{An almost-almost-Schur lemma}

With the above functional notation at hand, we can state our main stability result for inequality \eqref{eq:ADT_intro_functional}.

\begin{theorem}[Quantitative stability for \eqref{eq:ADT_intro_functional}]\label{thm:rev}
	There is a constant $c_{ADT}>0$ such that for all $u\in C^\infty(\mathbb S^3)$ with $u>0$ and $\sigma_1(u)>0$ we have
	\begin{equation}
		\label{eq:quantstab}
		\frac{F_2[1]}{F_1[1]^2} - \frac{F_2[u]}{F_1[u^{-2}]^2}  \geq c_{ADT}
		\inf_{\lambda,\Psi} \left( \|\lambda \, (u)_{\Psi}-1\|_{W^{1,2}(\mathbb S^{3})}^2 + \|\lambda \, (u)_{\Psi}-1\|_{W^{1,4}(\mathbb S^{3})}^4 \right),
	\end{equation}
	where the infimum is taken over all $\lambda\in \R$ and M\"obius transformations $\Psi:\mathbb S^3\to\mathbb S^3$.
\end{theorem}
 Here $W^{1,p}=W^{1,p}(\mathbb S^3)$ denotes the Sobolev space with norm $\|\cdot\|_{W^{1,p}}\coloneqq (\|\nabla \cdot\|_p^p+\|\cdot\|_p^p)^{1/p}$.

Since \eqref{eq:quantstab} is a quantitative stability result for an inequality which, in the form of \eqref{eq:ADT_intro}, is itself a quantitative version of Schur's lemma, it makes sense {-- in the spirit of \cite{DeLellis2012} --} to refer to Theorem \ref{thm:rev} as an \emph{almost-almost-Schur lemma}.

\subsection{The reverse $\sigma_2$-curvature inequality} 
It should be carefully noted that \eqref{eq:ADT_intro_functional} is a \emph{reverse inequality} in the sense that the {total $\sigma_k$-curvature} with the largest $k$ is on the \emph{smaller} side of the inequality. {Indeed, for $d>2k$, the total $\sigma_k$-curvature bounds (up to a constant factor) the total $\sigma_l$-curvature from above for all $0\leq l<k$; see \cite{Guan2004}. (Note that the total $\sigma_0$-curvature is set to be the volume.)} The model representative of such inequalities is the \emph{reverse $\sigma_2$-{curvature} inequality}
\begin{equation}
\tag{$\sigma_2$}
    \label{eq:reverse_sigma2}
    F_2[u] \leq F_2[1] \qquad \text{ for all } u \in C^\infty(\mathbb S^3) \, \text{ with $u > 0$ and $\sigma_1(u) > 0$ }
\end{equation}
due to Guan, Viaclovsky, and Wang \cite{Guan2003b}. The equality cases coincide with the ones of \eqref{eq:ADT_intro_functional}. Their original proof in \cite{Guan2003b} assumed the additional constraint $\sigma_2^{g}> 0$. Although Ge and Wang \cite[Theorem 1]{Ge2013} showed that this condition can be removed, the price to pay is that the Yamabe invariant has to stay bounded after removal. For $d=3$, this remained a long-standing open problem, which has recently been solved; see the forthcoming preprint \cite{Ge2025}.

Analogously to Theorem \ref{thm:rev}, we have the following stability result.  

\begin{theorem}[Quantitative stability for \eqref{eq:reverse_sigma2}]\label{thm:revsigma}
    There is a constant $c_{\sigma_2}>0$ such that for all $u\in C^\infty(\mathbb S^3)$ with $u>0$ and $\sigma_1(u)>0$ we have
	\begin{equation}
		\label{eq:quantstabsigma}
		F_2[1] - F_2[u] \geq c_{\sigma_2}
		\inf_{\lambda,\Psi} \left( \|\lambda \, (u)_{\Psi}-1\|_{W^{1,2}(\mathbb S^{3})}^2 + \|\lambda \, (u)_{\Psi}-1\|_{W^{1,4}(\mathbb S^{3})}^4 \right),
	\end{equation}
	where the infimum is taken over all $\lambda\in \R$ and M\"obius transformations $\Psi:\mathbb S^3\to\mathbb S^3$.
\end{theorem}
We first make a few comments on this statement.
 \begin{remarks}
\label{remarks_sigma2}
\begin{enumerate}[(i)]
\item \label{remark_monotonicity} We notice that \eqref{eq:reverse_sigma2} and the Sobolev inequality \eqref{eq:Sob} imply \eqref{eq:ADT_intro_functional}. As a consequence, also the stability result from Theorem \ref{thm:revsigma} for \eqref{eq:reverse_sigma2} implies Theorem \ref{thm:rev} via \eqref{eq:Sob}, with a constant $c_{ADT} = F_1[1]^{-2} c_{\sigma_2}$. However, instead of proving stability for \eqref{eq:reverse_sigma2} first, we follow a different approach and prove Theorem \ref{thm:revsigma} using Theorem \ref{thm:rev} as explained in the next subsection; see also Remark \ref{rem:sigma2ADT}. 

\item For dimension $d \geq 5$, the stability of the (non-reverse) $\sigma_2$-curvature inequality $\mathcal F_2[g] \geq \mathcal F_2[1]$ has been studied in the recent preprint \cite{Frank2024b}. Similarly to the results in \cite{Frank2024b}, this refinement of the reverse $\sigma_2$-curvature inequality is invariant under M\"obius transformations, and its corresponding Euler--Lagrange equation is fully non-linear. While our proof follows the overall scheme from \cite{Frank2024b}, heavy modifications and additional care are needed throughout to deal with the negativity of the exponent in the conformal factor $u^{-8}$ and of the prefactor of the first summand of $e_2(u)$ in \eqref{eq:e2(u)_definition}.

\item The exponent $2$ {of} the $W^{1,2}$-norm on the right side of \eqref{eq:quantstabsigma} is sharp. This can be proved like in \cite[Section 5]{Frank2024b}, and we omit a detailed proof. (Note that the parameter $\xi_\varepsilon$ of the M\"obius transformation $\Psi=\Psi_\varepsilon$ in \cite[Section 5]{Frank2024b} that minimizes $\|(1+\varepsilon\phi)_{\Psi}-1\|_{W^{1,2}}$ tends to $0$. In particular, we do not face any problems due to blow-up of $(1)_{\Psi_\varepsilon}$ as in Proposition \ref{prop:distcomp}.) On the other hand, we stress that it remains an open problem to determine whether the exponent $4$ of the $W^{1,4}$-norm in \eqref{eq:quantstabsigma} is optimal. While an affirmative answer can be expected by analogy with \cite{Frank2024b}, it appears that the family of functions used in \cite{Frank2024b} does not yield the conclusion for \eqref{eq:quantstabsigma} due to complications arising from the negative exponents.
\end{enumerate}
\end{remarks}

\subsection{An interpolation family of reverse inequalities and their stability}

Extending and systemizing Remark \ref{remarks_sigma2}.(\ref{remark_monotonicity}), we now explain how an entire family of reverse inequalities, as well as their quantitative stability, can be obtained by interpolation with the Sobolev inequality \eqref{eq:Sob}. The strongest inequality of this family, and hence one endpoint of the interpolation, is given by the inequality 
\begin{equation}
\tag{$\sigma_2$-$\sigma_1$}
    \label{eq:sigma2sigma1-ineq}
    F_2[u] F_1[u^{-2}] \leq F_2[1] F_1[1] \qquad \text{ for all $u > 0$ with $\sigma_1(u) > 0$\,.}
\end{equation} 
Equality holds in the same cases as in \eqref{eq:ADT_intro_functional} and \eqref{eq:reverse_sigma2}. The validity of this inequality was left as an open question in \cite{Ge2013} and has {recently been} proved in the forthcoming preprint \cite{Ge2025}. We refer to \eqref{eq:sigma2sigma1-ineq} as the \emph{$\sigma_2$-$\sigma_1$-curvature inequality}.

Applying \eqref{eq:sigma2sigma1-ineq} together with \eqref{eq:Sob}, we obtain the family 
\begin{equation}
\label{eq:interpol}
    F_2[u] F_1[u^{-2}]^{1 - \vartheta} \leq F_2[1] F_1[1]^{1 - \vartheta} \qquad \text{ for all $u > 0$ with $\sigma_1(u) > 0$}\,,
\end{equation}
with interpolation parameter $\vartheta \in [0, \infty)$. Moreover, for all $\vartheta$, equality holds if {and only if} $u = \lambda (1)_\Psi$ for $\lambda>0$ and $\Psi$ Möbius transformation. By \eqref{eq:Sob}, a smaller value of $\vartheta$ corresponds to a stronger inequality. Notice also that the three special reverse inequalities discussed so far, namely \eqref{eq:ADT_intro_functional}, \eqref{eq:reverse_sigma2}, and \eqref{eq:sigma2sigma1-ineq}, are all embedded into this family as the cases $\vartheta = 3$, $\vartheta = 1$, and $\vartheta = 0$, respectively. (The Sobolev inequality \eqref{eq:Sob} corresponds to $\vartheta \to \infty$.) 

From our stability result for \eqref{eq:ADT_intro_functional}, Theorem \ref{thm:rev}, we can deduce an analogous stability result for all inequalities from the family \eqref{eq:interpol}. 

\begin{corollary}[Interpolated stability]
    \label{cor:interpolation_stability}
    For all $\vartheta > 0$ and all $u \in C^\infty(\mathbb S^3)$ with $u > 0$ and $\sigma_1(u) > 0$, we have 
    \begin{equation*}
   F_2[1] F_1[1]^{1 - \vartheta} - F_2[u] F_1[u^{-2}]^{1 - \vartheta} \geq c(\vartheta)  \inf_{\lambda,\Psi} \left( \|\lambda \, (u)_{\Psi}-1\|_{W^{1,2}(\mathbb S^{3})}^2 + \|\lambda \, (u)_{\Psi}-1\|_{W^{1,4}(\mathbb S^{3})}^4 \right),
    \end{equation*}
    where  $c(\vartheta) \coloneqq  \min \{c_{ADT}F_1[1]^{3-\vartheta}, \vartheta c_{ADT}F_1[1]^{3-\vartheta/3}/3, F_2[1]F_1[1]^{1-\vartheta}/(2 |\mathbb S^3|) \}$, with $c_{ADT}$ being the constant from Theorem~\ref{thm:rev}.
\end{corollary}

It is remarkable that this produces stability results not only for inequalities between \eqref{eq:ADT_intro_functional} and \eqref{eq:reverse_sigma2} but also for all other inequalities up until, but excluding the endpoint cases $\vartheta = 0$ and $\vartheta\to\infty$. We emphasize that our proof of the almost-almost-Schur lemma, Theorem \ref{thm:rev}, is independent of  \eqref{eq:sigma2sigma1-ineq}, which only enters the proof of Corollary \ref{cor:interpolation_stability} (and hence Theorem \ref{thm:revsigma}).

In higher dimensions, interpolated nonlinear Yamabe-type inequalities similar to \eqref{eq:interpol} can be derived along with stability results as in Corollary \ref{cor:interpolation_stability} (excluding the endpoint cases). Nevertheless, the local bound in Section \ref{sec:local} exploits the uniform convergence of minimizing sequences, which is characteristic of $d=3$, in an essential way. This prevents us from extending our stability analysis, which avoids the subtle frequency decomposition carried out in \cite{Frank2024b}, to higher dimensions.

\subsection{Some context and related works} In this subsection we supplement our results with underlying theory and background material on conformal geometry (with a focus on the three-dimensional case) as well as stability of functional inequalities (with a focus on non-quadratic and reverse Sobolev-type inequalities). We close with a short discussion on a more general framework in the context of the almost-Schur lemma.

For a more {comprehensive}
background discussion of stability inequalities, we refer to the recent lecture notes \cite{Frank2024}. A detailed discussion related to {the stability of the $\sigma_2$-curvature inequality in higher dimensions} can be found in \cite{Frank2024b}.

\subsubsection*{Conformal geometry in lower dimensions}
Let $(M,g)$ be a $d$-dimensional Riemannian manifold with $d\geq 3$. The $\sigma_k^g$-curvature, $1\leq k\leq d$, is defined as the $k$-th elementary symmetric polynomial of the eigenvalues of the Schouten tensor with respect to the metric $g$. 
 In analogy to the well-known Yamabe problem and its solution (see \cite{Lee1987}, for instance), 
 the \textit{$\sigma_k$-Yamabe problem} consists in finding a metric $g$ conformally equivalent to a given metric $g_0$ such that the constant $\sigma_k$-curvature equation holds. {To guarantee ellipticity, it is common to assume $\sigma_l^g>0$ for all $l\leq k$; see \cite{Viaclovsky2000}, for instance.}  
While the constant $\sigma_k$-curvature equation is a semilinear equation in the conformal factor for $k=1$, it becomes fully-nonlinear for $k\geq 2$. 

If the dimension is small, then less information on the $\sigma_k$-curvatures is needed to fully characterize a manifold. Hence, a stronger form of rigidity is to be expected. As it turns out, the value $d=2k$ is critical, and for $d\leq 2k$ manifolds $(M,g)$ can be almost fully characterized by assuming $\sigma_l^g>0$ for all $l\leq k$.

Indeed, Guan, Viaclovsky, and Wang \cite{Guan2003b} showed that such $g$ have positive Ricci curvature. As a corollary, in the compact, locally conformally flat case, the manifold $(M,g)$ is conformally equivalent to a spherical space form; see \cite[Proposition 5]{Guan2003b} and \cite[Theorem 1 (B)]{Guan2004}. In case $k=2$ and $d=3$, an even stronger characterization holds: For $(M,g)$ merely compact and with nonnegative total $\sigma_2$-curvature, the critical points of $\mathcal F_2$ are given by metrics of constant sectional curvature; see \cite{Gursky2001}. If the manifold is compact, not conformally equivalent to a spherical space form, and $d<2k$, then Gursky and Viaclovsky \cite{Gursky2007} proved existence and regularity of solutions to the inhomogeneous $\sigma_k$-Yamabe problem and compactness of the set of solutions.  Hence, the reduction from Ricci curvature to scalar curvature bounds in dimension $3$ and $4$ by Ge and Wang in \cite{Ge2012,Ge2013b} seems to be inherently related to the nature of the problem; for more on the study of $4$-manifolds, we refer to \cite{Chang2002, Chang2002a}.

\subsubsection*{Stability of the Sobolev and reverse Sobolev-type inequalities} The question of (quantitative) stability was first raised in \cite{BREZIS198573} for the Sobolev inequality on $\R^d$. Bianchi and Egnell \cite{Bianchi1991} gave an affirmative answer to this problem -- along with a robust two-step method -- by bounding the deficit functional from below by the \textit{square} of the $\dot W^{1,2}(\R^d)$-distance to the set of optimizers.

The next natural question is to extend this result to the $p$-Sobolev inequality with $p\neq2$. After preliminary works in
\cite{Cianchi2009,Figalli2019,Neumayer2019}, Figalli and Zhang \cite{Figalli2022}  proved stability with a distance in terms of the gradient $L^p$-norm and with an optimal power $\max\{2,p\}$; see also \cite{Liu2025} for a stability result for critical points in the absence of bubbling. The non-quadratic stability exponents in \cite{Frank2024b, guerra2023sharp, Wang2025, Frank2025} are of a similar origin: They lack an inner product induced by the larger side of the inequality. 
Theorem \ref{thm:rev}, Theorem \ref{thm:revsigma}, and Corollary~\ref{cor:interpolation_stability} extend this notion to the setting of reverse Sobolev-type inequalities, which we discuss next. 

In the setting of the fractional Sobolev inequality on $\dot W^{s,2}(\R^d)$, $s<d/2$, for instance, quadratic stability was obtained in \cite{Chen2013}. In dimensions lower than $2s$, it was shown in  \cite{Hang2007, Frank2022} that the sign of the Sobolev inequality changes, leading to a notion of \textit{reverse} Sobolev inequality. This extends the previously known range of parameters for the fractional Sobolev inequality to include $s-d/2\in (0,1)\cup (1,2)$. The phenomenon of sign reversion for lower dimensions resembles the one found by Guan and Wang in \cite{Guan2004} for the $\sigma_k$-$\sigma_l$-curvature inequalities. Sharp stability for the reverse Sobolev inequality was proved by the first author in \cite{Koenig2025} for the full parameter regime $s-d/2\in (0,1)\cup(1,2)$. Gong, Yang and Zhang \cite{Gong2025} also obtained stability results for $s-d/2\in (1,2)$ based on a novel correspondence between the reverse Sobolev and the reverse HLS-inequality established in \cite{Dou2015}; see also \cite{Ngo2017, Carrillo19}.

\subsubsection*{Towards a more general stability result for the almost-Schur lemma} Our result covers a very special case of the Schur lemma. Indeed, while the Schur lemma is formulated for general Riemannian manifolds, the almost-Schur lemma is stated for manifolds with nonnegative Ricci curvature \cite{DeLellis2012}, and in the special case $d=3,4$ with nonnegative scalar curvature \cite{Ge2012,Ge2013b}. In turn, we confined ourselves to the $3$-sphere. We think it is an interesting problem to extend our stability result {for the almost-Schur lemma} to $\mathbb S^d$ (or more general manifolds). Further note that we assume our scalar curvature to be positive instead of nonnegative \cite{Ge2013b}, which seems reminiscent of the geometric background of the inequality rather than a technical obstruction of our method.

\subsection*{Acknowledgements} We would like to thank Alice Chang for raising the question of stability for the reverse $\sigma_2$-curvature inequality, and Rupert Frank for telling us about this problem. We are grateful to Guofang Wang for pointing out reference \cite{ChowLuNi2006} and informing us about the forthcoming work \cite{Ge2025}. 
Part of this work has been done while J.W.P.~was visiting Goethe University in Frankfurt, and he would like to thank T.K.~and the institute for their hospitality. T.K.~is funded by the Deutsche Forschungsgemeinschaft (DFG, German Research Foundation) – project number 555837013. Partial support through the DFG grants FR 2664/3-1 and TRR 352-Project-ID 470903074 and through the Studienstiftung des
deutschen Volkes (J.W.P.) is acknowledged.

\section{Proof strategy}

To prove stability of   \eqref{eq:ADT_intro_functional} as stated in Theorem \ref{thm:rev}, we apply the two-step method as promoted by Bianchi and Egnell in \cite{Bianchi1991}. This allows us to conclude the stability of the other reverse inequalities, Corollary \ref{cor:interpolation_stability}, in particular the stability of  \eqref{eq:reverse_sigma2} in Theorem \ref{thm:revsigma}, via interpolation.

Here and in the following, for $u \in C^\infty(\mathbb S^3)$ with $u > 0$ and \emph{for any $p \in \R \setminus \{0\}$}, we abbreviate $$\|u\|_p \coloneqq \left(\int_{\mathbb S^3} u^p \,\mathrm d\omega\right)^\frac{1}{p}\,.$$
Moreover, {the infimum} $\inf_\Psi$ (respectively, $\inf_{\lambda, \Psi})$ is always understood to be taken over all Möbius transformations $\Psi: \mathbb S^3 \to \mathbb S^3$ (and $\lambda \in \R$), unless stated otherwise.

\subsection{The Bianchi--Egnell strategy} Working in the framework of \cite{Bianchi1991}, we have to prove two propositions: A global-to-local reduction and a local bound.

\begin{proposition}[Global-to-local reduction]\label{prop:glob2locrev}
	Let $(u_j)\subset C^\infty(\Sph^3)$ be a sequence of positive functions with $\sigma_1(u_j)>0$ for all $j$ and satisfying, as $j\to\infty$, 
	\begin{equation*}
	    \frac{F_{2}[u_j]}{F_1[u^{-2}_j]^2}\to  \frac{F_{2}[1]}{F_1[1]^2}
	\qquad\text{and}\qquad
	\| u_j \|_{-12} \to \| 1 \|_{-12} \,.
	\end{equation*}
	Then
	$$
	\inf_{\Psi} \| (u_j)_{\Psi} - 1 \|_{W^{1,4}(\Sph^3)} \to 0
	\qquad\text{as}\ j\to\infty \,.
	$$
\end{proposition}	
Due to the negative exponent in the conformal factor for $d=3$, we develop a regularization trick for optimizing sequences of \eqref{eq:ADT_intro_functional} and a blow-up criterion in Lemma \ref{lem:hölder-infty} and \ref{lem:blowup}, respectively, which turn out to hold (and are stated) for general dimensions.

\begin{proposition}[Local bound]\label{prop:locrev}
	There is a constant $c>0$ with the following property: Let $(u_j)\subset C^\infty(\Sph^3)$ be a sequence of positive functions with $\sigma_1(u_j)>0$ for all $j$, with $\| u_j \|_{-12} = \| 1\|_{-12}$ for all $j$ and with $\inf_{\Psi} \| (u_j)_{\Psi} - 1 \|_{W^{1,4}(\Sph^3)} \to 0$ as $j\to\infty$. Then
	$$
	\liminf_{j\to\infty} \frac{F_2[1]F_{1}[1]^{-2}-F_{2}[u_j] F_{1}[u_j^{-2}]^{-2}}{\inf_{\Psi} \left( \| (u_j)_{\Psi} - 1 \|_{W^{1,2}(\Sph^3)}^2 + \| (u_j)_{\Psi} - 1 \|_{W^{1,4}(\Sph^3)}^4 \right)} \geq c \,.
	$$
\end{proposition}

The proof of these two propositions will take up the bulk of the paper. More precisely, we prove Proposition \ref{prop:glob2locrev} in Section \ref{sec:2}. Sections \ref{sec:orth} and \ref{sec:local} are in turn devoted to the proof of Proposition \ref{prop:locrev}.

\subsection{Proof of the main results}
With Propositions \ref{prop:glob2locrev} and \ref{prop:locrev} at hand, the proof of Theorem~\ref{thm:revsigma} follows by contradiction.

It is convenient to abbreviate
\begin{equation}
    \label{eq:dist}
    \dist(u)\coloneqq \inf_{\lambda,\Psi} \left( \|\lambda \, (u)_{\Psi}-1\|_{W^{1,2}(\mathbb S^{3})}^2 + \|\lambda \, (u)_{\Psi}-1\|_{W^{1,4}(\mathbb S^{3})}^4 \right).
\end{equation} 

\begin{proof}[Proof of Theorem \ref{thm:revsigma}]
   Assume by contradiction that a sequence $(u_j)\subset C^\infty(\Sph^d)$ satisfies $u_j>0$ and $\sigma_1(u_j)>0$ for all $j$ and
	\begin{equation}
		\label{eq:thmproofass}
		\frac{F_2[1]F_{1}[1]^{-2}-F_{2}[u_j] F_{1}[u_j^{-2}]^{-2}}{\dist(u_j)} \to 0. 
	\end{equation}
	exists. Since the quotient is $0$-homogeneous in $u_j$, we may assume that $\| u_j \|_{-12} = \| 1\|_{-12}$ for all $j$. Since $\dist(u_j) \leq 2 |\mathbb S^3|$ (by choosing $\lambda=0$),
	 \eqref{eq:thmproofass} implies that $F_{2}[u_j] F_{1}[u_j^{-2}]^{-2}\to F_2[1]F_{1}[1]^{-2}$ as $j\to\infty$. Proposition \ref{prop:glob2locrev} then gives $\inf_{\Psi} \| (u_j)_{\Psi} - 1 \|_{W^{1,4}(\Sph^3)} \to 0$ as $j\to\infty$. Choosing $\lambda=1$ as a competitor for the infimum in \eqref{eq:dist}, Proposition \ref{prop:locrev} is applicable, which leads to a positive, $j$-independent lower bound for the quotient in \eqref{eq:thmproofass} and thus to a contradiction.
\end{proof}

Corollary \ref{cor:interpolation_stability} follows from Theorem \ref{thm:rev} by an interpolation argument involving the Sobolev inequality \eqref{eq:Sob} and the $\sigma_2$-$\sigma_1$-{curvature} inequality \eqref{eq:sigma2sigma1-ineq}. 

\begin{proof}[Proof of Corollary \ref{cor:interpolation_stability}]
First note that by choosing $\lambda = 0$, we find $\dist(u) \leq 2 |\mathbb S^3|$. Thus, when $F_2[u] \leq 0$, we have  
  \[ F_2[1] F_1[1]^{1 - \vartheta} - F_2[u] F_1[u^{-2}]^{1 - \vartheta} \geq  F_2[1] F_1[1]^{1 - \vartheta} \geq c(\vartheta) \dist(u)\,.  \]
   Thus, we may assume  $F_2[u] > 0$ in the following.

    For $\vartheta = 3$, the statement is just Theorem \ref{thm:rev}.
  
    Suppose now that $\vartheta > 3$. Then, using \eqref{eq:Sob}, we have $F_1[u^{-2}]^{3-\vartheta}\leq F_1[1]^{3 - \vartheta}$, and hence
    \begin{align*}
        F_2[1] F_1[1]^{1 - \vartheta} - F_2[u] F_1[u^{-2}]^{1 - \vartheta}  &= F_1[1]^{3 - \vartheta} \left( \frac{F_2[1]}{F_1[1]^2} - \frac{F_2[u]}{F_1[u^{-2}]^2} \left( \frac{F_1[u^{-2}]}{F_1[1]} \right)^{3 - \vartheta} \right) \\
        &\geq  F_1[1]^{3 - \vartheta} \left( \frac{F_2[1]}{F_1[1]^2} - \frac{F_2[u]}{F_1[u^{-2}]^2}  \right)  \geq c_{ADT} F_1[1]^{3 - \vartheta} \dist(u)\,,
    \end{align*}
where we used Theorem \ref{thm:rev} for the last inequality. 

Next suppose that $\vartheta < 3$. By Theorem \ref{thm:rev}, we have 
\[ 0 <\frac{F_2[u]}{F_1[u^{-2}]^2}  \leq \frac{F_2[1]}{F_1[1]^2}  - c_{ADT} \dist(u),  
\]
in particular $c_{ADT} F_1[1]^2\dist(u) < F_2[1]$. Using this together with the concavity of $t \mapsto t^{\frac \vartheta 3}$, and inequality \eqref{eq:sigma2sigma1-ineq}, we obtain 
\begin{align*}
    F_2[u] F_1[u^{-2}]^{1 - \vartheta} & \leq (F_2[u] F_1[u^{-2}])^{1-\frac \vartheta 3}\left(\frac{F_2[1]}{F_1[1]^2}  - c_{ADT} \dist(u)\right)^{\frac{\vartheta}3} \\
    & \leq (F_2[1] F_1[1])^{1 - \frac \vartheta 3} \left(\frac{F_2[1]}{F_1[1]^2}\right)^{\frac{\vartheta}3} \left( 1- \vartheta   \frac{c_{ADT}F_1[1]^2}{3F_2[1]} \dist(u)\right) \\
    &= F_2[1] F_1[1]^{1 - \frac\vartheta 3}  \left(1 - \vartheta   \frac{c_{ADT}F_1[1]^2}{3F_2[1]} \dist(u) \right).
\end{align*}
As a consequence, we find
\begin{equation*}
    F_2[1] F_1[1]^{1 - \vartheta} - F_2[u] F_1[u^{-2}]^{1 - \vartheta} 
     \geq  F_2[1] F_1[1]^{1 - \frac \vartheta 3} \vartheta   \frac{c_{ADT}F_1[1]^2}{3F_2[1]} \dist(u)= \vartheta\frac {c_{ADT}} 3  F_1[1]^{3 - \frac\vartheta 3} \dist(u)\,. 
\end{equation*}
This completes the proof. 
\end{proof}

Finally, as already mentioned, Theorem \ref{thm:revsigma} is the special case $\vartheta = 3$ of Corollary \ref{cor:interpolation_stability}.

\section{Global-to-local reduction}\label{sec:2}

{In this section our goal is to prove the first step of the Bianchi--Egnell method.} 

In case $d=3$, the sign of the terms containing derivatives changes in $F_2[u]$, and 
the functional $F_2$ is not bounded from below anymore but bounded from above, while keeping the same set of minimizers -- constant functions up to M\"obius transformations. 
In \cite{Frank2024b} a monotonicity result by Guan and Wang \cite{Guan2004} was used in order to reduce the analysis to the compactness properties of minimizing sequences for the Sobolev inequality. A similar key role in the proof of Proposition \ref{prop:glob2locrev} will be played by the inequality \eqref{eq:reverse_sigma2} (or indeed any inequality of the family \eqref{eq:interpol} with $\vartheta < 3$).

\begin{proof}[Proof of Proposition \ref{prop:glob2locrev}]
	Consider $(u_j)\subset C^\infty(\Sph^3)$ with $u_j>0$ and $\sigma_1(u_j)>0$ for all $j$ that satisfy
	$$
	\frac{F_{2}[u_j]}{F_1[u^{-2}_j]^2}\to  \frac{F_{2}[1]}{F_1[1]^2}
	\qquad\text{and}\qquad
	\| u_j \|_{-12} \to \| 1\|_{-12}
	$$ as $j\to\infty$.
    To prove the proposition, it suffices to show that there is a sequence $(\Psi_j)$ of M\"obius transformations such that
	\begin{equation}
		\label{eq:glob2locgoal}
		(u_j)_{\Psi_j}\to 1
		\qquad\text{in}\ W^{1,4}(\Sph^3) \,.
	\end{equation}
	If we show for an arbitrary subsequence of $(u_j)_{\Psi_j}$ that a further subsequence satisfies this convergence, then the conclusion of Proposition \ref{prop:glob2locrev} holds for the whole sequence. Thus, we {can} pass to a subsequence without loss of generality. We further assume $\|u_j\|_{-12}^{-12}=\|1\|_{-12}^{-12}=|\mathbb S^3|$ by scaling invariance of $F_{2}$ and $F_1$.

We have
\[ 1 = \lim_{j \to \infty} \frac{F_{2}[u_j]}{F_1[u^{-2}_j]^2} \frac{F_1[1]^2}{F_{2}[1]} \leq \limsup_{j \to \infty} \frac{F_2[u_j]}{F_2[1]} \cdot  \limsup_{j \to \infty} \frac{F_1[1]^2}{F_1[u_j^{-2}]^2} \leq 1,  \]
where the last inequality follows by \eqref{eq:reverse_sigma2} and \eqref{eq:Sob}. This chain of inequalities implies $F_1[u_j^{-2}] \to  F_1[1]$ and $F_2[u_j]\to F_2[1]$. Thus, if we define the positive functions $$w_j\coloneqq u_j^{-2}$$ on $\Sph^3$,
	then $$F_1[w_j]  \to F_1[1]\quad \text{as} \ j\to\infty\qquad \text{and}\qquad \| w_j\|_6^6 = \|u_j\|_{-12}^{-12} = |\Sph^3|
	\quad\text{for all}\ j \,.$$
	
    {Next, we apply} the classification of optimizers \cite{Rodemich1966, Aubin1976, Talenti1976} and Lions's concentration compactness \cite{Lions1985,Lions1985a} for the Sobolev inequality on $\R^3$. {After} translating it via stereographic projection to $\mathbb S^3$ and passing to a subsequence if necessary, there are M\"obius transformations $(\Psi_j)$ such that $[w_j]_{\Psi_j}\to 1$ in $W^{1,2}(\Sph^3)$ {as $j\to\infty$.} Let us set 
	$$
	\Tilde w_j \coloneqq [w_j]_{\Psi_j} 
	\qquad\text{and}\qquad \Tilde u_j \coloneqq (u_j)_{\Psi_j} \,.
	$$
	
   To deduce $\Tilde u_j\to 1$ in $W^{1,4}(\Sph^3)$ from $\Tilde w_j \to 1$ in $W^{1,2}(\Sph^3)$, we note that by conformal invariance
	\begin{equation}
		\label{eq:glob2locproof}
		F_2[1] + o(1) = F_{2}[u_j] = F_{2}[\Tilde u_j] = |\mathbb S^3|^{\frac13} \int_{\Sph^3}\left(\frac34 \Tilde u_j^4-f_j\right)\,\mathrm d\omega\,,
	\end{equation} where
	\begin{equation}
	    \label{eq:fj}f_j \coloneqq 64 \left(\sigma_1(\Tilde u_j)+\frac{1}{2} |\nabla \Tilde u_j|^2+\frac{1}{32} \Tilde u_j^2\right) |\nabla \Tilde u_j|^2\,.
	\end{equation}
	
   In the last step of \eqref{eq:glob2locproof}, we used $\|u_j\|_{-12}^{-12} = |\Sph^3|$ for all $j$.
      Since $\Tilde w_j \to 1$ in $W^{1,2}(\mathbb S^3)$, and thus pointwise almost everywhere along a subsequence, we infer that $\Tilde u_j = \Tilde w_j^{-1/2} \to 1$ pointwise almost everywhere. Thus, $(\Tilde u_j)$ satisfies the hypotheses of Lemma  \ref{lem:uj_bounded} below, and we deduce that $(\Tilde u_j)$
    converges uniformly to $1$ after possibly passing to a subsequence.
    In particular, we have $\Tilde u_j\to 1$ in $L^4(\Sph^3)$ and, consequently,
	$$
	\frac{3}{4}|\mathbb S^3|^{\frac13} \int_{\Sph^3}  \Tilde u_j^4\,\mathrm d\omega = \frac{3}{4}|\mathbb S^3|^{\frac43}   + o(1)  =F_2[1] + o(1) \,.
	$$
	Thus, \eqref{eq:glob2locproof} leads to
	$$
	\int_{\Sph^3} f_j \,\mathrm d\omega = o(1) \,.
	$$
	Since $f_j$ is a sum of nonnegative terms, $|\nabla \Tilde u_j|$ tends to $0$ in $L^4(\Sph^3)$, which proves \eqref{eq:glob2locgoal} for normalized subsequences and hence completes the proof.
\end{proof}

{The following lemma is a peculiarity of optimizing sequences of the three-dimensional $\sigma_2$-curvature inequality \eqref{eq:reverse_sigma2}  when comparing it with its higher dimensional versions.}
It describes how the pointwise convergence of an optimizing sequence can be upgraded to uniform convergence. 

\begin{lemma}
    \label{lem:uj_bounded}
    Let {$(u_j)\subset C^\infty(\Sph^3)$} with $u_j > 0$ and $\sigma_1(u_j)>0$ satisfy $F_2[u_j]\to F_2[1]$ for $j\to \infty$. Then for all $j$ sufficiently large, one has
    \begin{equation}
        \label{eq:W14L4bound}
        \int_{\Sph^3} |\nabla u_j|^4 \, \mathrm d \omega \leq {\frac{3}{128}} \int_{\Sph^3} u_j^4 \, \mathrm d \omega\,. 
    \end{equation} 
    If in addition for almost every $\omega \in \Sph^3$ the sequence $(u_j(\omega))$ is bounded, then $(u_j)$ is in fact bounded in $W^{1,4}(\Sph^3)$ and converges uniformly along a subsequence.  
\end{lemma}
The bound \eqref{eq:W14L4bound} is very strong since it says that the $W^{1,4}$-norm (and hence, by Morrey's embedding, the $C^{0, 1/4}$-norm) of any minimizing sequence is equivalent to its $L^4$-norm. As the proof will show, this is a consequence of the mixed signs in $e_2(u)$ coming from the reverse setting.  

\begin{proof}
    Since $F_2[1] > 0$, {we see thanks to \eqref{eq:glob2locproof} from the previous proof} that 
    \[ 0 <  F_{2}[u_j] =   |\mathbb S^3|^{\frac13} \int_{\Sph^3}\left(\frac34 u_j^4-f_j\right)\,\mathrm d\omega\,, \]
    for all $j$ large enough, {where $f_j$ is defined as in \eqref{eq:fj} but with $u_j$ instead of $\Tilde u_j$.} As $f_j$ is a sum of nonnegative terms, this yields 
    \[ \frac{3}{4} \int_{\Sph^3} u_j^4 \,\mathrm d\omega> \int_{\Sph^3} f_j \,\mathrm d\omega\geq 32 \int_{\Sph^3} |\nabla u_j|^4\,\mathrm d\omega\,, \]
    which is \eqref{eq:W14L4bound}. 

    To prove the second part of the lemma, we note that $(u_j)$ satisfies \eqref{eq:asshölderbound} by \eqref{eq:W14L4bound} together with Morrey's and Hölder's inequalities. Thus, Lemma \ref{lem:hölder-infty} below yields that $(u_j)$ is bounded in $L^\infty(\Sph^3)$ and uniformly convergent along a subsequence. Applying once more Hölder's inequality and \eqref{eq:W14L4bound}, it follows that  $(u_j)$ is bounded in $W^{1,4}(\Sph^3)$. 
\end{proof}

Now we are left to deduce the conclusion of the previous lemma from a H\"older-space version of \eqref{eq:W14L4bound}. The next lemma then concludes the global-to-local reduction. {It holds for general dimensions.}

\begin{lemma}
    \label{lem:hölder-infty}
    Let {$d \in \N$} and let $(u_j) \subset C^\infty(\Sph^d)$ with $u_j > 0$. Suppose that for almost every $\omega \in \Sph^d$ the sequence $(u_j)(\omega)$ is bounded. Suppose {further} that for some $\alpha \in (0,1]$ and $C > 0$ we have
    \begin{equation}
        \label{eq:asshölderbound}
        [u_j]_{C^{0,\alpha}(\Sph^d)}  \leq C \|u_j\|_{L^\infty(\Sph^d)}\,. 
    \end{equation}
    Then $u_j$ is uniformly bounded in $L^\infty(\Sph^d)$. In particular, $u_j$ converges uniformly along a subsequence. 
\end{lemma}

\begin{proof}
    By contradiction, if the conclusion is not true, then there are $(\omega_j) \subset \Sph^d$ such that 
    \[ u_j(\omega_j) \geq \frac{1}{2} \|u_j\|_\infty \to \infty \qquad \text{ as } j \to \infty\,.  \]
    By taking a subsequence, we may assume $\omega_j \to \omega_\infty$ as $j \to \infty$ for some $\omega_\infty \in \Sph^d$. By assumption, there is {$\omega_0\in \mathbb S^d$} with $0< |\omega_\infty - \omega_0|^\alpha \leq {(4C)^{-1}}$ such that $(u_j(\omega_0))$ is uniformly bounded. Assumption \eqref{eq:asshölderbound} gives that 
    \[ \frac{u_j(\omega_j) - u_j(\omega_0)}{|\omega_j - \omega_0|^\alpha} \leq [u_j]_{C^{0,\alpha}(\Sph^d)}  \leq C \|u_j\|_{L^\infty(\Sph^d)} \leq 2C u_j(\omega_j)\,, \]
    or equivalently
    \[ 1 - \frac{u_j(\omega_0)}{u_j(\omega_j)} \leq 2 C |\omega_j - \omega_0|^\alpha\,. \]
    Letting $j \to \infty$, the choice of $\omega_0$ yields $1 \leq \frac{1}{2}$, a contradiction. Hence, $(u_j)$ is uniformly bounded in $L^\infty(\Sph^d)$.
    
    Together with \eqref{eq:asshölderbound}, the Arzela--Ascoli theorem then provides uniform convergence of a subsequence of $(u_j)$.
\end{proof}

\section{Orthogonality conditions and comparability between $W^{1,2}$ and $W^{1,4}$} 
\label{sec:orth}
{In this section we make some preparations to prove the second step of the Bianchi--Egnell method. More specifically, we show that} the error term arising from the distance minimization in \eqref{eq:dist} can be chosen to be `approximately' orthogonal to spherical harmonics of degree $0$ and $1$ {with respect to the $L^2$-inner product} while still vanishing in the $W^{1,4}$-norm at the same time.

\begin{proposition} \label{prop:distcomp} 
	Let $(u_j)\subset W^{1,4}(\Sph^3)$ with $\| u_j \|_{-12} = \|1\|_{-12}$ for all $j$ and $\inf_\Psi \|(u_j)_\Psi -1\|_{W^{1,4}}\to 0$ as $j\to\infty$. Then there is a sequence $(\Psi_j)$ of M\"obius transformations such that
	$$
	r_j \coloneqq (u_j)_{\Psi_j} - 1
	$$
	satisfies, for all sufficiently large $j$,
	\begin{equation}
		\label{eq:distanceboundsrev}
			\| r_j \|_{W^{1,2}} = \inf_\Psi \|(u_j)_\Psi -1\|_{W^{1,2}}
		\qquad\text{and}\qquad
		\| r_j \|_{W^{1,4}} \lesssim \inf_\Psi \|(u_j)_\Psi -1\|_{W^{1,4}}
	\end{equation}
	as well as
	\begin{equation}\label{eq:orthorev}
		\left|	\int_{\mathbb S^3} r_j \, \mathrm d \omega \right| \lesssim \|r_j\|^{2}_2  \qquad	\text{and} \qquad
		\left| \int_{\mathbb S^3} \omega_i \, r_j \, \mathrm d \omega \right| \lesssim \|r_j\|^2_{W^{1,2}} \,, \ i=1,\dots,4\,.
	\end{equation}
\end{proposition}
Here and in the following, we use $\lesssim$ to indicate that the left side is bounded by the right side
up to a {universal constant (unless stated otherwise). In Lemma \ref{lem:hölder-infty} and \ref{lem:blowup} the constants are allowed to depend on the dimension.} The symbol $\gtrsim$ is defined analogously. 

Since the conformal factor has a negative exponent, the effect of {a}  M\"obius transformation on a function differs drastically. Therefore, we lack a global bound on the $W^{1,4}$-norm of smooth functions under M\"obius transformations as given in \cite[Lemma 6]{Frank2024b}, which is crucial for the proof of Proposition \ref{prop:distcomp}. However, we can use a blow-up criterion and a local version of \cite[Lemma 6]{Frank2024b} to overcome this problem; see Lemma \ref{lem:blowup}.

{The following definitions are stated for general dimensions $d\in \N$ as Lemma \ref{lem:blowup} remains valid in this generality.} We will use the same parametrization of M\"obius transformations by rotations $A\in O(d+1)$ elements $\xi\in B_1(0)$, the unit ball in $\R^{d+1}$, as in \cite{Frank2024b} or \cite{Frank2025}, given by
\begin{equation*}
	\Psi(\omega)\coloneqq A\Psi_\xi(\omega)\,,\qquad \Psi_\xi(\omega)
	\coloneqq \frac{(1-|\xi|^2)\omega -2(1-\xi\cdot\omega)\xi}{1-2\xi\cdot\omega + |\xi|^2}
	\,,\qquad \omega \in \mathbb S^d \,.
\end{equation*}
A short computation shows that
\begin{equation*}
	(J_{\Psi_\xi}(\omega))^\frac 1d = \frac{1-|\xi|^2}{1-2\xi\cdot\omega + |\xi|^2}\qquad \text{and}\qquad\Psi_\xi^{-1}(\omega) = \Psi_{-\xi}(\omega)\,, \qquad\omega\in\mathbb S^d\,.
\end{equation*}
    \begin{proof}
    We start with the almost orthogonality conditions \eqref{eq:orthorev}. For $|\tau|<1/2$ we have the universal bound
    $$|(1+\tau)^{-12}-1-\tau|\lesssim |\tau|^2\,.$$
    By uniform convergence, we have $|r_j|<1/2$ for {all} $j$ large enough, which we can assume after passing to a subsequence if necessary.
    After integrating over the sphere, the normalization condition gives
    $$\left|\int_{\mathbb S^3}r_j\,\mathrm d\omega\right|=\left|\|1+r_j\|_{-12}^{-12}-\|1\|_{-12}^{-12}+\int_{\mathbb S^3}r_j\,\mathrm d\omega\right| \lesssim \|r_j\|_2^2\,, $$ which verifies the first estimate in \eqref{eq:orthorev}. The proof for the second one is given in \cite[Lemma 10]{Frank2024b} and continues to hold for three dimensions.

      Since the bound \cite[Lemma 6]{Frank2024b} is not available, we have to prove \eqref{eq:distanceboundsrev} differently. Nevertheless, {the observation that the former estimates blow up in fact facilitates a more direct approach.} Let us quickly show how this is done.

  First, if we parametrize $\Psi=A\Psi_\xi$ as mentioned after Proposition \ref{prop:distcomp}, we can use $\Psi=A\Psi_\xi=\Psi_{A\xi}\circ A$ and a change of variables $\omega\mapsto A^{-1}\omega$ to find 
\begin{equation}\label{eq:inf}
   \inf_\Psi \|(u)_\Psi - 1 \|_{W^{1,p}}=\inf_\xi \|(u)_{\Psi_\xi}- 1 \|_{W^{1,p}}\, \end{equation}
 for every smooth $u>0$ and every $p \in [1, 4]$.

  As the quantity $\| (u)_{\Psi_\xi} - 1\|_{W^{1,p}}$ blows up by Lemma \ref{lem:blowup} for $|\xi|\to 1$, the previous infimum is attained.

   To prove the bound in \eqref{eq:distanceboundsrev}, we assume by contradiction that there are sequences $(u_j)\subset W^{1,4}(\Sph^3)$ and $(\Psi_j'),(\Psi_j'')$ of M\"obius transformations, which attain the infimum in \eqref{eq:inf} {with $u=u_j$ for $p=2$ and $p=4$,} respectively, such that
	$$
	\|(u_j)_{\Psi_j'} - 1 \|_{W^{1,4}} > j \|(u_j)_{\Psi_j''} - 1 \|_{W^{1,4}} \,.
	$$

    Parametrizing $(\Psi_j'')^{-1}\circ \Psi_j'$ by $\tilde A_j\in O(4)$ and $\tilde \xi_j\in B_1(0)$, we are going to show next that $\limsup_{j\to\infty}|\tilde \xi_j|<1$. Note that by assumption $\|(u_j)_{\Psi_j'} - 1 \|_{W^{1,2}}\lesssim \|(u_j)_{\Psi_j''} - 1 \|_{W^{1,4}}\to 0$ as $j\to\infty$. Hence, we know that {the sequences} $(\|(u_j)_{\Psi_j'}\|_{W^{1,2}})$ and $(\|(u_j)_{\Psi_j''}\|_{W^{1,4}})$ are bounded uniformly in $j$. If we bound $$\sup_j\|((u_j)_{\Psi_j'})_{((\Psi_j'')^{-1}\circ \Psi_j')^{-1}}\|_{W^{1,2}}=\sup_j\|(u_j)_{\Psi_j''}\|_{W^{1,2}}\lesssim \sup_j\|(u_j)_{\Psi_j''}\|_{W^{1,4}}<\infty\,,$$
    we deduce by Lemma \ref{lem:blowup} with $v_j=(u_j)_{\Psi_j'}$ and $\xi_j=-\tilde \xi_j$ that $\limsup_{j\to\infty}|\tilde \xi_j|<1$. The rotation $\tilde A_j$ can be removed once more by a change of variables.
    Therefore, we can apply the conformal bound locally with $\xi=\tilde \xi_j$ as in \eqref{eq:confbddloc2} {below} together with the assumption to obtain
that    \begin{equation*}
		\|(u_j)_{\Psi'_j}-1\|_{W^{1,4}} \gtrsim j 	\|(u_j)_{\Psi'_j}-(1)_{(\Psi''_j)^{-1}\circ \Psi'_j}\|_{W^{1,4}} \,. 
	\end{equation*}
	Observe that
	$$
	v_j\coloneqq \frac{(u_j)_{\Psi_j'}-1}{\|(u_j)_{\Psi_j'}-1\|_{W^{1,4}}}\,, \qquad\tilde v_j\coloneqq \frac{(u_j)_{\Psi'_j}-(1)_{(\Psi''_j)^{-1}\circ \Psi'_j}}{\|(u_j)_{\Psi_j'}-1\|_{W^{1,4}}}\,, \qquad\Delta_j\coloneqq v_j-\tilde v_j 
	$$ 
	satisfy, as $j\to\infty$,
	\begin{equation*}
		\|v_j\|_{W^{1,4}}=1\,,\qquad \|\tilde v_j\|_{W^{1,4}}\to 0\,, \qquad\|\Delta_j\|_{W^{1,4}}\to 1 \,.
	\end{equation*} Bounding $v_j$ in the $W^{1,2}$-norm as
	$$
	\|v_j\|_{W^{1,2}} = \frac{\inf_\Psi \|(u_j)_\Psi - 1 \|_{W^{1,2}}}{\|(u_j)_{\Psi_j'}-1\|_{W^{1,4}}} 
	\leq \frac{\|(u_j)_{\Psi''_j}-1\|_{W^{1,2}}}{\|(u_j)_{\Psi_j'}-1\|_{W^{1,4}}}\lesssim \frac{\|(u_j)_{\Psi''_j}-1\|_{W^{1,4}}}{\|(u_j)_{\Psi'_j}-1\|_{W^{1,4}}}\leq \frac{1}{j}\to 0
	$$ 
	gives
	\begin{equation*}
		\|\tilde v_j\|^2_{W^{1,2}}+2\langle \tilde v_j, \Delta_j\rangle_{W^{1,2}}+\|\Delta_j\|^2_{W^{1,2}}=	\| v_j\|^2_{W^{1,2}}\to 0 
	\end{equation*} as $j\to\infty$. As the notation suggests, $\langle \cdot,\cdot\rangle_{W^{1,2}}$ denotes the inner product in $W^{1,2}(\mathbb S^3)$. We know that $\|\tilde v_j\|_{W^{1,2}}\lesssim \|\tilde v_j\|_{W^{1,4}}\to 0$, so $\|\tilde v_j\|^2_{W^{1,2}}+2\langle \tilde v_j, \Delta_j\rangle_{W^{1,2}} \to 0$, and therefore
	$$
	\|\Delta_j\|^2_{W^{1,2}} \to 0 \,.
	$$
	
	 Up to the factor $\|(u_j)_{\Psi'_j}-1\|_{W^{1,4}}^{-1}$, the function $\Delta_j$ is given by $(1)_{(\Psi''_j)^{-1}\circ \Psi'_j} - 1$. Using a change of variables as before, we find that $$\|\Delta_j\|_{W^{1,p}}=\|(1)_{\Psi_{\tilde A_j\tilde \xi_j}} - 1\|_{W^{1,p}}$$ for $p=2,4$. We can think of the latter function as being a function of a variable $\tilde A_j\tilde \xi_j\in B_1(0)$.
	Since $\liminf_{j\to\infty}|\tilde A_j\tilde \xi_j|<1$, all norms of such functions of $\tilde A_j\tilde \xi_j$ are equivalent along the corresponding subsequence. Thus, the properties $\|\Delta_j\|_{W^{1,4}}\to 1$ and $\|\Delta_j\|_{W^{1,2}} \to 0$ contradict each other.
	\end{proof}

	\medskip
The following lemma is not restricted to $d=3$ and Sobolev exponent $4$. For general $d\in \N$ and $q>1$, the action of a M\"obius transformation $\Psi$ on a function $u$ can be written as $$(u)_{\Psi,q}\coloneqq (1)_{\Psi,q}u\circ\Psi \coloneqq J_\Psi^{\frac{d-q}{qd}}u\circ\Psi\,.$$
In the following lemma, we suppress the additional index $q$ for the sake of readability.

\begin{lemma}\label{lem:blowup} Let $d\in \N$, $p\geq 1$, and $q>d$ with $q\geq p$. For a sequence of smooth functions $(v_j)$ that is uniformly bounded in $W^{1,q}(\mathbb S^d)$ and satisfies $\|v_j\|_{-qd/(q-d)}^{-qd/(q-d)}=|\mathbb S^d|$, we have
$$\limsup_{j\to\infty}\|(v_j)_{\Psi_{\xi_j}}\|_{W^{1,p}}= \infty\qquad \text{if and only if}\qquad \limsup_{j\to\infty}|\xi_j|=1\,,$$ where $\Psi_{\xi_j}$ denotes the M\"obius transformation corresponding to $A=\mathbbm 1$ and $\xi_j\in B_1(0)$. Moreover,
\begin{equation}\label{eq:confbddloc2}
    \|(f)_{\Psi_{\xi_j}}\|_{W^{1,p}}\lesssim \|f\|_{W^{1,p}}
\end{equation}
holds for $\limsup_{j\to\infty}|\xi_j|<1$ and any smooth function $f$ on $\mathbb S^d$.
\end{lemma}

\begin{proof} By Morrey's inequality and compact embedding of H\"older spaces, it is well-known that $v_j\to v\in C^{0,\gamma}(\mathbb S^d)$ with $\gamma\in (0,1-d/q)$. Before pursuing the proof of the statement, let us first show that $v>0$.

Assume by contradiction that $v(S)=0$ for some $S\in \mathbb S^d$. Since $(v_j)$ is bounded in $W^{1,q}(\mathbb S^d)$, we deduce by Morrey's inequality that
$$ |v_j(\omega)|{+o_{j\to\infty}(1)} = |v_j(\omega)-v_j(S)|\lesssim |\omega-S|^{1-\frac{d}{q}}\,,$$
for every $\omega \in \mathbb S^d$. 
Fatou's lemma then implies {that}
$$\liminf_{j\to \infty} \int_{\mathbb S^d} {|v_j(\omega)|}^{-\frac{qd}{q-d}}\,\mathrm d\omega
\gtrsim \int_{\mathbb S^d} |\omega-S|^{-d}\,\mathrm d\omega=\infty\,. $$
This, however, contradicts our normalization in the statement of the lemma. Therefore, $v>0$.

To prove the equivalence, we distinguish two cases. If $\limsup_{j\to\infty}|\xi_j|=1$, then along a subsequence $\xi_j\to\xi\in \mathbb S^d$. We also see that $\Psi_{\xi_j}(\omega)\to -\xi$ 
uniformly for all $\omega \in \mathbb S^d$ with $|\xi-\omega|>\varepsilon$
for arbitrary but fixed $\varepsilon>0$. Therefore, we can estimate
\begin{align}\notag
    \|(v_j)_{\Psi_{\xi_j}}\|^p_{W^{1,p}}&\geq\|(1)_{\Psi_{\xi_j}}v_j\circ\Psi_{\xi_j}\|^p_p\geq \left(\frac{\varepsilon^2}{1-|\xi_j|^2}\right)^{\frac p q (q-d)}\int_{\{|\xi_j-\omega|>\varepsilon\}}|v_j(\Psi_{\xi_j}(\omega))|^p\,\mathrm d\omega\\&= \left(\frac{\varepsilon^2}{1-|\xi_j|^2}\right)^{\frac p q (q-d)}|\{|\xi_j-\omega|>\varepsilon\}| \,| v(-\xi)|^p \,(1+o_{j\to\infty}(1))\,.\label{eq:blowupbound}
\end{align} The last step follows by using the uniform convergence of $v_j\to v$ on $\mathbb S^d$ and $\Psi_{\xi_j}(\omega)\to -\xi$ on $\{|\xi-\omega|>\varepsilon\}$. Since $v>0$, the second line in \eqref{eq:blowupbound} tends to $\infty$ as $j\to \infty$.

If $\limsup_{j\to\infty}|\xi_j|<1$, then we can bound the Jacobian $J_{\Psi_{\xi_j}}=(1)_{\Psi_{\xi_j}}^{-{qd/(q-d)}}$ of $\Psi_{\xi_j}$ from above and below by some positive, $j$-independent constant. Similarly, we can bound the derivative of the Jacobian from above. Note that $$|\nabla (v_j\circ \Psi_{\xi_j})|=|(d\Psi_{\xi_j})^T((\nabla v_j)\circ \Psi_{\xi_j})|= J_{\Psi_{\xi_j}}^{\frac1d}|(\nabla v_j)\circ \Psi_{\xi_j}|$$ by conformality of $\Psi_{\xi_j}$, where $(d(\Psi_{\xi_j})_\omega)^{\rm T}:T_{\Psi_{\xi_j}(\omega)}\Sph^d \to T_\omega\Sph^d$ denotes the adjoint of the map $d(\Psi_{\xi_j})_\omega:T_\omega\Sph^d \to T_{\Psi_{\xi_j}(\omega)}\Sph^d$ with respect to the given inner products on these spaces. Hence, the bounds on the Jacobian and the transformation formula give  $$\|(v_j)_{\Psi_{\xi_j}}\|^p_{W^{1,p}}\leq\|(1)_{\Psi_{\xi_j}}v_j\circ\Psi_{\xi_j}\|_{p}^p +\|\nabla (1)_{\Psi_{\xi_j}} v_j\circ\Psi_{\xi_j}\|_{p}^p+\left\|(1)^{-{\frac{d}{q-d}}}_{\Psi_{\xi_j}} (\nabla v_j)\circ\Psi_{\xi_j} \right\|^p_{p}\lesssim \|v_j\|_{W^{1,p}}^p$$ {with a constant depending on $\varepsilon$.} As $v_j$ is bounded uniformly in $j$ in $W^{1,p}(\mathbb S^d)$ by assumption, we proved the converse.

In addition, for the constant sequence $v_j=f$ this implies \eqref{eq:confbddloc2}.
\end{proof}

The first case in the proof actually shows that
$$\limsup_{j\to\infty}|\xi_j|=1\Rightarrow\limsup_{j\to\infty}\|(v_j)_{\Psi_{\xi_j}}\|_{p}=\infty\Rightarrow\limsup_{j\to\infty}\|(v_j)_{\Psi_{\xi_j}}\|_{W^{1,p}}=\infty\,,$$ so we can always use the $L^p$-norm instead of the $W^{1,p}$-norm in the equivalence in Lemma \ref{lem:blowup}.

\section{Local bound}
\label{sec:local}
In this section our goal is to prove the second step of the Bianchi--Egnell method. In contrast to the $\sigma_2$-curvature inequality for $d>4$ in \cite{Frank2024b}, our local analysis for \eqref{eq:ADT_intro_functional} does not suffer from a spectral gap issue, which streamlines the frequency decomposition. Moreover, our argument enjoys uniform convergence of the remainder $r_j$.
\begin{remark} \label{rem:sigma2ADT}
       As a consequence of Corollary \ref{cor:interpolation_stability}, we deduce stability for the $\sigma_2$-curvature inequality \eqref{eq:reverse_sigma2} as given in Theorem \ref{thm:revsigma}. However, we note that a local stability analysis of \eqref{eq:reverse_sigma2} would lead to Theorem \ref{thm:revsigma} directly and would give rise to an alternative proof (with a different stability constant) of Corollary \ref{cor:interpolation_stability}, and in particular Theorem \ref{thm:rev}. We omit the details.
  \end{remark}
  
\begin{proof}[Proof of Proposition \ref{prop:locrev}]
Proposition \ref{prop:distcomp} provides conformal transformations $(\Psi_j)$ such that $$(u_j)_{\Psi_j}=1+r_j$$ with $r_j\to 0$ in $W^{1,4}(\mathbb S^d)$ and {the approximate} orthogonality conditions \eqref{eq:orthorev}. Thanks to the conformal invariance of \eqref{eq:ADT_intro_functional}, we can drop the transformations $\Psi_j$ in the following. 

We first expand our functional inequality in terms of $r_j$. For this purpose, let us set
$$E_2[u]\coloneqq \int_{\mathbb S^3} \left(\frac34 u^4-64 \left(\sigma_1(u)+\frac{1}{2} |\nabla u|^2+\frac{1}{32} u^2\right) |\nabla u|^2\right) \,\mathrm d\omega\,, $$
$$E_1[u^{-2}]\coloneqq 2\int_{\mathbb S^3} \left(|\nabla (u^{-2})|^2+\frac34 (u^{-2})^2\right)\,\mathrm d\omega = 8\int_{\mathbb S^3} u^{-6}\left(|\nabla u|^2+\frac3{16} u^{2}\right)\,\mathrm d\omega$$ 
for $u\in C^\infty(\mathbb S^3)$.
 Since $r_j\to 0$ uniformly and $1+r_j>0$, we can expand $u_j^p$, $p\in\R\setminus\{0\}$, under the integral sign to arbitrary order in $r_j$. A similar argument holds for $|\nabla u_j|^2u_j^p$ as $u_j$ is bounded in $W^{1,2}(\mathbb S^3)$. 

Expanding to fourth order leads to 
\begin{align*} \|u_j\|_{-12}^{-12}=&\int_{\mathbb S^3}\left(1-12r_j+78r_j^2
    \right)\,\mathrm d\omega+\mathcal O(\|r_j\|_2^2\|r_j\|_\infty+\|r_j\|_\infty^5)\,,
   \\  E_1[u_j^{-2}]=& \ 8\int_{\mathbb S^3}\left(|\nabla r_j|^2+\frac 3{16}\left(1-4r_j+10r_j^2
   \right)\right)\,\mathrm d\omega+\mathcal O(\|r_j\|_{W^{1,2}}^2\|r_j\|_\infty+\|r_j\|_\infty^5)\,.
\end{align*}
Note that we used $\|r_j\|^3_3\leq \|r_j\|^2_2\|r_j\|_\infty$ and $\|r_j\|^4_4\leq \|r_j\|^2_2\|r_j\|_\infty^2$ to dismiss third and fourth order terms in $r_j$ here.
Taking the appropriate powers and using $|\int_{\mathbb S^3} r_j\,\mathrm d\omega|\lesssim \|r_j\|_2^2$ gives
\begin{align*}
 \|u_j\|_{-12}^{12}=&\ \frac{1}{|\mathbb S^3|}\bigg(1-
 \frac{1}{|\mathbb S^3|}\int_{\mathbb S^3}\left(-12r_j+78r_j^2
 \right)\,\mathrm d\omega
 \bigg)
+\mathcal O(\|r_j\|_2^2\|r_j\|_\infty+\|r_j\|_\infty^5)\,,
  \\  E_1[u_j^{-2}]^2=& \frac{9}{4}|\mathbb S^3|^2\left(1
    +\frac{2}{|\mathbb S^3|}\int_{\mathbb S^3}\left(\frac{16}{3}|\nabla r_j|^2-4r_j+10r_j^2
    \right)\,\mathrm d\omega\right)+\mathcal O(\|r_j\|_{W^{1,2}}^2\|r_j\|_\infty+\|r_j\|_\infty^5)\,.
\end{align*}
Multiplying both expansions, we find
\begin{align*}
   E_1[u_j^{-2}]^2\|u_j\|_{-12}^{12} = &\ 
    24\|\nabla r_j\|_2^2
    +\frac{9}{4}\bigg(|\mathbb S^3|+\int_{\mathbb S^3}\left(4r_j-58r_j^2
    \right)\,\mathrm d\omega
 \bigg)+\mathcal O(\|r_j\|_{W^{1,2}}^2\|r_j\|_\infty+\|r_j\|_\infty^5)\,.
\end{align*}
Along the same lines, we compute
\begin{align*}E_2[u_j]&\coloneqq \frac{3}{4}|\mathbb S^3|+\int_{\mathbb S^3} \left(3r_j +\frac{9}{2}r_j^2-64 \,\sigma_1(1+r_j)|\nabla r_j|^2-32|\nabla r_j|^4-2 |\nabla r_j|^2\right)  \,\mathrm d\omega\\&+\mathcal O(\|r_j\|_{W^{1,2}}^2\|r_j\|_\infty)\,, \end{align*}
Using $\sigma_1(1+r_j)>0$, we eventually obtain
\begin{equation}\label{eq:lowerbound}
     \frac{1}{3}E_1[u_j^{-2}]^2\|u_j\|_{-12}^{12}-E_2[u_j]
    \geq \ 10\|\nabla r_j\|_2^2-48\|r_j\|_2^2
    +32\|\nabla r_j\|_4^4+\mathcal O(\|r_j\|_{W^{1,2}}^2\|r_j\|_\infty+\|r_j\|_\infty^5)\,.
\end{equation}

Our final goal is to bound the right side from below by \begin{equation}\label{eq:lowerbound2}
(\|r_j\|_{W^{1,2}}^2+\|r_j\|_{W^{1,4}}^4)(1+o(1))
\end{equation} 
up to multiplication by a constant. Indeed, $r_j=(u_j)_{\Psi_j}-1$ is a competitor for the infimum in Proposition \ref{prop:locrev}, so we can therefore conclude the desired local bound
\begin{equation*}
	\liminf_{j\to\infty} \frac{F_2[1] F_1[1]^{-2} - F_2[u_j]F_1[u_j^{-2}]^{-2}}{\inf_\Psi \left( \| (u_j)_\Psi - 1 \|_{W^{1,2}}^2 + \| (u_j)_\Psi - 1 \|_{W^{1,4}}^4 \right)} 
\gtrsim\liminf_{j\to\infty}\frac{\frac{1}{3}E_1[u_j^{-2}]^2\|u_j\|_{-12}^{12}-E_2[u_j]}{\| r_j\|_{W^{1,2}}^2
    +\|r_j\|_{W^{1,4}}^4}
    \gtrsim 1
\end{equation*}
for $j$ large enough using $\| u_j \|_{-12}^{-12}=|\Sph^3|$.
Note that the error in \eqref{eq:lowerbound} is of lower order since by Morrey's inequality
$$\frac{\|r_j\|_{W^{1,2}}^2\|r_j\|_\infty+\|r_j\|_\infty^5}{\| r_j\|_{W^{1,2}}^2+\|r_j\|_{W^{1,4}}^4}\lesssim \|r_j\|_\infty\to 0$$ as $j\to \infty$.

To find a lower bound of the form \eqref{eq:lowerbound2}, we study the quadratic term in \eqref{eq:lowerbound} given  by $$Q[r_j]\coloneqq 10\|\nabla r_j\|_2^2-48\|r_j\|_2^2\,.$$
Thanks to Proposition \ref{prop:distcomp}, $r_j$ is approximately orthogonal to the spherical harmonics of degree $0$ and $1$, and thus $Q$ admits a spectral gap (up to terms of higher order) that guarantees a lower bound of the form $Q[r_j]\gtrsim\|r_j\|_{W^{1,2}}^2+\mathcal O(\|r_j\|_{W^{1,2}}^2\|r_j\|_\infty+\|r_j\|_\infty^5)$.

More concretely, we split $r_j$ into spherical harmonics of degree $0$, $1$, and $\geq 2$, that is, 
$$r_j=r_j^{(0)}+r_j^{(1)}+\tilde r_j$$
with
$$
r_j^{(0)} \coloneqq \Pi_0 r_j= \frac{1}{|\mathbb S^3|}\langle r_j, 1 \rangle_2\,,\qquad
r_j^{(1)} \coloneqq \Pi_1 r_j=\frac{4}{|\mathbb S^3|}\sum_{i=1}^4\langle r_j, \omega_i \rangle_2 \omega_i\,,
\qquad
\tilde r_j \coloneqq \sum_{\ell=2}^\infty \Pi_\ell r_j \,,
$$ where $\langle \cdot,\cdot\rangle_2$ denotes the inner product in $L^2(\mathbb S^3)$ and $\Pi_\ell$ the $L^2(\mathbb S^3)$-orthogonal projection onto spherical harmonics of degree $\ell \in \N_0$. Recall that $(-\Delta) f=l(l+2)f$ for spherical harmonics of degree $l$ in three dimensions. For more details on spherical harmonics, we refer to \cite[p.~137--152]{Stein1990}. Due to the orthogonality of spherical harmonics, $Q[r_j]$ disintegrates into  $Q[r_j^{(0)}]+Q[r_j^{(1)}]+Q[\tilde r_j]$. Then we apply approximate orthogonality, Proposition \ref{prop:distcomp}. First, degree $0$ gives
$$Q[r_j^{(0)}]=-48|\mathbb S^3|^{-1}\langle r_j, 1 \rangle_2^2=\mathcal O(\|r_j\|_2^4)\,.$$
Next, as $(-\Delta) \omega_i=3\omega_i$, $i=1,\dots,4$, degree $1$ gives
$$Q[r_j^{(1)}]=
-\frac{288}{|\mathbb S^3|^2}\sum_{i=1}^4 \langle r_j, \omega_i \rangle_2^2\|\omega_i\|_2^2
=\mathcal O(\|r_j\|_{W^{1,2}}^4)\,.$$
Finally, we apply the spectral gap of the Laplacian on $\mathbb S^3$ for spherical harmonics of degree $2$ and higher, which gives
$$Q[\tilde r_j]\geq \frac{32}9\|\tilde r_j\|_{W^{1,2}}^2\,. $$

Recovering the missing frequencies in the form of constant multiples of $\|r_j^{(0)}\|_{W^{1,2}}^2$, $\|r_j^{(1)}\|_{W^{1,2}}^2$, and $\|r_j\|_{4}^4$ from the error term in \eqref{eq:lowerbound}, it follows that
\begin{equation*}
     \frac{1}{3}E_1[u_j^{-2}]^2\|u_j\|_{-12}^{12}-E_2[u_j]
    \geq \frac{32}9\| r_j\|_{W^{1,2}}^2
    +32\|r_j\|_{W^{1,4}}^4+\mathcal O(\|r_j\|_{W^{1,2}}^2\|r_j\|_\infty+\|r_j\|_\infty^5)
    \,.
\end{equation*}
Since the right side is of the form \eqref{eq:lowerbound2}, this completes the proof.
\end{proof}




\newcommand{\etalchar}[1]{$^{#1}$}
\providecommand{\bysame}{\leavevmode\hbox to3em{\hrulefill}\thinspace}
\providecommand{\MR}{\relax\ifhmode\unskip\space\fi MR }
\providecommand{\MRhref}[2]{
	\href{http://www.ams.org/mathscinet-getitem?mr=#1}{#2}
}
\providecommand{\href}[2]{#2}

\end{document}